\documentclass[10pt]{amsart}
\usepackage[utf8]{inputenc}
\usepackage{amsmath}
\usepackage{amssymb}

\makeatletter
\@namedef{subjclassname@2010}{
 \textup{2010} Mathematics Subject Classification}
\makeatother

\newtheorem{Thm}{Theorem}[section]
\newtheorem{Prp}[Thm]{Proposition}
\newtheorem{Cor}[Thm]{Corollary}
\newtheorem{Lma}[Thm]{Lemma} 
\theoremstyle{definition}
\newtheorem{Def}[Thm]{Definition}
\newtheorem{Rk}[Thm]{Remark}

\numberwithin{equation}{section}

\newcommand{\gl}{\mathfrak{gl}}
\newcommand{\ad}{\mathrm{ad}}
\newcommand{\GL}{\mathrm{GL}}
\newcommand{\Ker}{\mathrm{Ker}}
\newcommand{\End}{\mathrm{End}}

\newcommand{\U}{\mathrm{U}}
\newcommand{\tr}{\mathrm{tr}}
\newcommand{\diag}{\mathrm{diag}}
\newcommand{\Gal}{\mathrm{Gal}}
\newcommand{\sll}{\mathfrak{sl}}
\newcommand{\su}{\mathfrak{su}}
\newcommand{\g}{\mathfrak{g}}
\newcommand{\h}{\mathfrak{h}}
\newcommand{\Ad}{\mathrm{Ad}}
\newcommand{\ch}{\mathrm{char}}
\newcommand{\Br}{\mathrm{Br}}
\newcommand{\Id}{\mathrm{Id}}
\newcommand{\Lie}{\mathrm{Lie}}
\newcommand{{\fb}}{\overline{F}}

\begin{document}
\title[Lie bialgebras and Hilbert's Seventeenth Problem]{Lie bialgebras, Fields of Cohomological Dimension at Most 2 and Hilbert's Seventeenth Problem}
\author[S.\ Alsaody and A.\ Stolin]{Seidon Alsaody and Alexander Stolin}
\address{Department of Mathematical Sciences\\ Chalmers University of Technology and the University of Gothenburg\\412 96 G\"oteborg\\ Sweden}

\begin{abstract} We investigate Lie bialgebra structures on simple Lie algebras of non-split type $A$. It turns out that there are several classes of such Lie bialgebra structures, and it is possible to classify some of them. The classification is obtained using Belavin--Drinfeld cohomology sets, which are introduced in the paper. Our description is particularly detailed over fields of cohomological dimension at most two, and is related to quaternion algebras and the Brauer group. We then extend the results to certain rational function fields over real closed fields via Pfister's theory of quadratic forms and his solution to Hilbert's Seventeenth Problem.
\end{abstract}

\maketitle

\section{Introduction}
The study of quantum groups was initiated by Kulish and Reshetikhin in \cite{KR} and developed independently by Drinfeld \cite{Dr} and Jimbo \cite{Ji} in the 1980s. Over the past three decades, the area has seen major activity in various directions. 

Quantum groups are deformations of universal enveloping algebras of Lie algebras. More specifically, if $F$ is a field of characteristic zero, then by a \emph{quantum group} we understand a topologically free Hopf algebra $U_\hbar$ over the ring $F[[\hbar]]$ of formal power series in $F$, such that, over $F$, the quotient $U_\hbar/\hbar U_\hbar$ is isomorphic to the universal enveloping algebra $U(\g)$ of some $F$-Lie algebra $\g$. In \cite{EK1} and \cite{EK2}, Etingof and Kazhdan constructed their quantization functors, thereby establishing an equivalence of categories that relates the problem of classifying quantum groups to that of classifying Lie bialgebras over $F[[\hbar]]$. If $\g$ is finite-dimensional, the problem can be reduced further to the classification of Lie bialgebra structures on the scalar extension $\g_{F((\hbar))}$ of $\g$ to the field $F((\hbar))$. This spurred the motivation to classify Lie 
bialgebras over fields of characteristic 
zero which are not algebraically closed.

Over algebraically closed fields, Lie bialgebra structures on simple Lie algebras have been classified by Belavin and Drinfeld \cite{BD}. Over non-closed fields, results have been obtained by Stolin and co-authors, upon introducing a cohomology theory known as \emph{Belavin--Drinfeld cohomology}. This descent-type method resembles that of Galois cohomology, and has been applied to various split Lie algebras over fields which are not algebraically closed. The aim of this paper is to extend it to non-split Lie algebras. We investigate the situation for such algebras of type $A$. 

More specifically, over a field $F$ of characteristic zero with a quadratic field extension $K=F(\sqrt{d})$, we consider the Lie algebra $\su(n,F,d)$ of all $A\in\sll(n,K)$ satisfying $\overline{A}^T+A=0$, where the conjugation $A\mapsto \overline{A}$ is induced by the non-trivial element of the Galois group $\Gal(K/F)$. We then ask when a Lie bialgebra structure on $\sll(n,K)$ descends to $\su(n,F,d)$, and study the behaviour of these structures. Over the algebraic closure $\fb$, any such structure is a coboundary Lie bialgebra, gauge equivalent to the coboundary of $\lambda r_{BD}$ for some $\lambda\in\fb$ and a non-skew symmetric $r$-matrix $r_{BD}$ in the Belavin--Drinfeld classification. We prove that there are three possibilities for $\lambda$; namely, up to a scalar multiple in $F$, we have $\lambda=1$, $\lambda=\sqrt{d}$ and $\lambda=\sqrt{d'}$ for some $d'\in K^*\setminus {K^*}^2$. We will refer to these three types of Lie bialgebra structures as \emph{basic}, \emph{quadratic} and \emph{
twisted}, respectively. In the quadratic and twisted case, we show that such Lie bialgebra structures exist only if $r_{BD}$ is essentially of Drinfeld--Jimbo type. Our investigation is particularly detailed in the quadratic case, where the Drinfeld double of the Lie bialgebra is $\sll(n,K)$ itself. There we achieve a classification of these Lie bialgebra structures for $r$-matrices of Drinfeld--Jimbo type, over fields of cohomological dimension at most 2, as well as over function fields in at most 2 indeterminates over real-closed fields.

The paper is organized as follows. In Section 2 we give the necessary preliminaries and then focus on preparing the setting for the definition and characterization of the necessary cohomology theory in the case where the Lie bialgebra is of quadratic Drinfeld--Jimbo type. We then derive the cohomology theory in Section 3. In Section 4 we use quaternion algebras to give a construction of certain cocycles over arbitrary fields of characteristic zero. This enables us to accomplish, in Section 5, a fairly explicit classification over several classes of fields, linking the problem to Pfister's theory of quadratic forms and Hilbert's Seventeenth Problem. In Section 6 we consider structures of basic type, and set up the Belavin--Drinfeld cohomology. We get a classification of Lie bialgebras of Drinfeld--Jimbo type, and structural results for the other types. In Section 7 we finally consider the twisted case, showing that Drinfeld--Jimbo Lie bialgebras are essentially the only ones possible, and deriving a 
cohomology theory in this case.

\section{Preliminaries}
\subsection{Lie Bialgebras and $r$-Matrices}
Let $G$ be a finite-dimensional reductive algebraic group over a field $F$ of characteristic zero, and $\g=(\g,[,])$ the semisimple part of $\Lie(G)$. A \emph{Lie coalgebra structure} on $\g$ is a map $\delta:\g\to\g\otimes\g$ such that its transpose is a Lie algebra structure on the dual space $\g^*$. If $\delta$ moreover satisfies the \emph{cocycle condition}
\[\delta([a,b])=(\ad_a\otimes 1+1\otimes\ad_a)\delta(b)-(\ad_b\otimes 1+1\otimes\ad_b)\delta(a)\]
for any $a,b\in \g$, then $(\g,[,],\delta)$ is called a \emph{Lie bialgebra}. Abusing notation, we will speak of $\g$ as a Lie bialgebra, suppressing the algebra and coalgebra structures whenever they are understood. Two Lie bialgebra structures $\delta_1$ and $\delta_2$ on $(\g,[,])$ are called \emph{equivalent} if there exists $\lambda\in F^*$ and $X\in G(F)$ such that
\[\delta_2=\lambda(\Ad_X\otimes \Ad_X)\delta_1,\]
and \emph{gauge equivalent} if this holds with $\lambda=1$. A Lie bialgebra $(\g,[,],\delta)$ is called a \emph{coboundary Lie bialgebra} if $\delta=\partial r$ for some $r\in\g\otimes \g$, viz.
\[\delta(a)=-(\ad_a\otimes 1+1\otimes\ad_a)r\]
for all $a\in\g$. Embedding into the universal enveloping algebra of $\g$, will often write this in the form
\[\delta(a)=[r,a\otimes 1+1\otimes a].\]
We will use the term (gauge) equivalent referring to $r$-matrices whose coboundary Lie bialgebras are (gauge) equivalent. Note that both equivalence and gauge equivalence depend on the field of scalars.

Over algebraically closed fields, it is known that every Lie bialgebra structure on a finite-dimensional simple Lie algebra $\g$ is a coboundary Lie bialgebra structure $\partial r$ where $r$ is an $r$-matrix, i.e.\ a solution to the classical Yang--Baxter equation $\mathrm{CYB}(r)=0$, which further satisfies that $r+r^{21}$ is $\g$-invariant, where $r^{21}=\sum b_i\otimes a_i$ whenever $r=\sum a_i\otimes b_i$. The Yang--Baxter operator $\mathrm{CYB}:\g^{\otimes 2}\to \g^{\otimes 3}$ is defined by
\[\mathrm{CYB}(r)=[r_{12},r_{13}]+[r_{12},r_{23}]+[r_{13},r_{23}],\]
where the commutator is that of the tensor power of the universal enveloping algebra, and the notation $r_{ij}$ is defined by setting e.g.\ $(a\otimes b)_{13}=a\otimes 1\otimes b$, and extending by linearity.

Belavin and Drinfeld achieved, in \cite{BD}, a classification of such $r$-matrices. Let $\g$ be a finite-dimensional simple Lie algebra over an algebraically closed field $F$, and fix a Cartan subalgebra $\h$ of $\g$ with an orthonormal basis $\{h_i\}$, a root system $\Delta$ and a subset $\Delta^+$ of positive roots. Write $e_{\alpha}$ for the Chevalley generator corresponding to $\alpha\in\Delta$, $\Omega$ for the Casimir element, and $\Omega_0$ for its $\h$-component. An \emph{admissible triple} is then a triple $(\Gamma_1,\Gamma_2,\tau)$, where $\Gamma_1$ and $\Gamma_2$ are subsets of the set $\Gamma$ of simple roots, and $\tau:\Gamma_1\to\Gamma_2$ is an isometric bijection such that for each $\alpha\in\Gamma_1$ there is $k\in\mathbb N$ such that $\tau^k(\alpha)\notin\Gamma_1$. The Belavin--Drinfeld classification then reads as follows.

\begin{Thm}\label{TBD} Every Lie bialgebra structure $\delta$ on $\g$ satisfies $\delta=\partial r$ for some $r$-matrix $r$. If $r$ is not skewsymmetric, then $r$ is equivalent to
\[r_{BD}=r_0+r_1\]
where $r_0\in\h\otimes\h$ satisfies $r_0+r_0^{21}=\Omega_0$, and for some admissible triple $(\Gamma_1,\Gamma_2,\tau)$,
\[r_1=\sum_{\alpha\in\Delta^+} e_{\alpha}\otimes e_{-\alpha}+\sum_{\alpha\in\mathrm{Span}(\Gamma_1)^+}\sum_{k\in\mathbb N} e_{\alpha}\wedge e_{-\tau^k(\alpha)}\]
  and
\[\forall \alpha\in\Gamma_1: (\tau(\alpha)\otimes\Id+\Id\otimes\alpha)(r_0)=0.\]
\end{Thm}

\begin{Rk}
We call any such $r$-matrix $r_{BD}$ a \emph{Belavin--Drinfeld $r$-matrix}. Whenever we write $r_{BD}=r_0+r_1$, this will refer to the decomposition above. 
\end{Rk}

For the most part, we shall focus on a particular such $r$-matrix, namely the \emph{standard} or \emph{Drinfeld--Jimbo} $r$-matrix which, in a sense, is the simplest case. Fixing $\h$, $\{h_i\}$, $\Delta$ and $\Delta^+$, and writing $e_{\alpha}$ as above, the associated Drinfeld--Jimbo $r$-matrix 
is
\[r_{DJ}=\frac{1}{2}\sum_i h_i\otimes h_i+\sum_{\alpha\in\Delta^+} e_{\alpha}\otimes e_{-\alpha}.\]

For the remainder of the paper, we fix a field $F$ of characteristic zero having a quadratic extension $K$, and fix $d\in F^*\setminus F^{*2}$ with $K=F(\sqrt{d})$. When additional assumptions are made on the field, these will be stated explicitly. All algebras and bialgebras are assumed to be finite-dimensional over their respective ground fields.

The non-identity element of the Galois group $\Gal(K/F)$ maps $z=a+b\sqrt{d}$ to $\overline z=a-b\sqrt{d}$ for any $a,b\in F$, and extends in the usual way to $\gl(n,K)$ for any positive integer $n$. We let $N:K\to F$ denote the norm function, which is given by $N(z)=\overline{z}z$. For any $n$, we consider the $F$-algebraic group whose group of rational points is
\[\U(n,d)(F)=\{X\in\GL(n,K): \overline{X}^TX=1\}\]
and the corresponding simple Lie algebra
\[\su(n,F,d)=\{A\in\sll(n,K):\overline A^T+A=0\}.\]
This generalizes the well-known $\su_n=\su(n,\mathbb R,-1)$.

We will henceforth write $\g$ for $\su(n,F,d)$. Extending scalars to $K$ and writing $\g_K$ the corresponding extension $K\otimes_F\g$ of $\g$, we see that $\g_K=\sll(n,K)$. 

Let $\delta$ be a Lie bialgebra structure on $\g$. Extending scalars to an algebraic closure ${\fb}$ containing $K$, we obtain a Lie bialgebra structure $\delta_{{\fb}}$ on $\g_{{\fb}}=\sll(n,{\fb})$. By the above there exists $\lambda\in{\fb}^*$, $X\in\GL(n,{\fb})$ and a Belavin--Drinfeld $r$-matrix $r_{BD}$ such that $\overline{\delta}=\partial r$ with
\[r=\lambda (\Ad_X\otimes\Ad_X)(r_{BD}).\]
However, the converse is not true, i.e.\ not every choice of $\lambda$, $X$ and $r_{BD}$ is such that $\partial r$ descends to a Lie bialgebra structure on $\g$. 

A first step is to determine those values of $\lambda$ for which $\partial r$ descends to $\g_K$. This was done in \cite{PS} and gives the following necessary conditions.

\begin{Thm}\label{TPS} Let $\delta_K$ be a Lie bialgebra structure on $\sll(n,K)$ with Drinfeld double $D$. Then $D$ satisfies precisely one of the following conditions.
\begin{enumerate}
\item $D\simeq \sll(n,K)\otimes K^2$. Then $\delta_K=\partial r$ where $r=\lambda (\Ad_X\otimes\Ad_X)(r_{BD})$ for some $\lambda\in K^*$, $X\in \GL(n,\overline{F})$ and with $r_{BD}$ a Belavin--Drinfeld $r$-matrix over $\overline{F}$. 
\item $D\simeq \sll(n,K)\otimes K(\sqrt{d'})$ for some $d'\in K^*\setminus {K^*}^2$. Then $\delta_K=\partial r$ where $r=\lambda\sqrt{d'} (\Ad_X\otimes\Ad_X)(r_{BD})$ for some $\lambda\in K^*$, $X\in \GL(n,\overline{F})$ and with $r_{BD}$ a Belavin--Drinfeld $r$-matrix over $\overline{F}$.
\item $D\simeq \sll(n,K)\otimes K[\epsilon]$ with $\epsilon^2=0$.
\end{enumerate}
\end{Thm}

We refer to \cite{ES} or \cite{PS} for the definition of the Drinfeld double of a Lie bialgebra.

Throughout, we will use the notation $X^*=\overline{X}^T$ and $(X\otimes Y)^*=\overline{X}^T\otimes \overline{Y}^T$ for any $X,Y\in \gl(n,K)$, which we extend to $\gl(n,K)\otimes \gl(n,K)$.

\begin{Rk}\label{Rlambda} An $r$-matrix $s\in\g_{\fb}\otimes\g_{\fb}$ defines a Lie bialgebra structure on $\g$ if and only if for all $a\in \g$ there exists $b\in\g\otimes\g$ such that
\[b=[s,1\otimes a+a\otimes 1].\] 
Since $b$ and $a$ are invariant under the action induced by any $\sigma\in\Gal(\fb/K)$, we find that for any such $\sigma$, $s-\sigma(s)$ commutes with $1\otimes a+a\otimes 1$ for all $a\in\g$. Noting further that $b^*=b$ and applying $^*$ to both sides, we get
\[b=[1\otimes (-a)+(-a)\otimes 1,s^*]=[s^*,1\otimes a+a\otimes 1].\] 
Thus $s-s^*$ commutes with $1\otimes a+a\otimes 1$ for all $a\in\g$ as well. Since $g_{\fb}$ has an $\fb$-basis consisting of elements in $\g$, this implies that
\[\begin{array}{lll}
 \forall \sigma\in\Gal(\fb/K): s-\sigma(s)\in\fb\Omega& \text{and}& s-s^*=\in\fb\Omega.
  \end{array}
 \]

Conversely, if $s$ satisfies these two conditions, then for all $a\in\g$, $[s,1\otimes a+a\otimes 1]$ is invariant under any $\sigma\in\Gal(\fb/K)$  and under $^*$. Using the fact that $g_{\fb}$ has an $\fb$-basis consisting of elements in $\g$, one then deduces that $[s,1\otimes a+a\otimes 1]\in\g\otimes\g$
\end{Rk}

To determine the possible values of $\lambda$ for which one also has a Lie bialgebra structure on $\g$, the following result from \cite{KKPS} is useful.

\begin{Lma}\label{LK} If $r$ and $r'$ are non-skewsymmetric $r$-matrices satisfying $r+r^{21}=\lambda\Omega$ and $r'=r-\mu\Omega$, then either $\mu=0$ or $\mu=\lambda$. 
\end{Lma}

The above remark asserts that if $\delta$ is a Lie bialgebra structure on $\g$ such that $\delta_{\fb}=\partial r$, then $r$ satisfies
\begin{equation}\label{rstar}
\begin{array}{lll}
r^*=r-\mu\Omega & \text{and}&r_{21}^*=r_{21}-\mu\Omega  
\end{array}
\end{equation}
for some $\mu\in{\fb}$. Theorem \ref{TPS} then implies that, since in particular $\partial r$ should define a Lie bialgebra structure on $\g_K$, then either $\mu\in K$ or $\mu=j$ for some $j\in{\fb}$ satisfying $j^2\in K$ and $j\notin K$. The next result refines this under the condition that $\partial r$ defines a Lie bialgebra structure on $\g$.

\begin{Prp}\label{Plambda} Assume that $\delta$ is a Lie bialgebra structure on $\g$ with $\delta_{\fb}=\partial r$. Then $r=\lambda(\Ad_X\otimes\Ad_X)(r_{BD})$, where $r_{BD}$ is a Belavin--Drinfeld $r$-matrix, $X\in\GL(n,{\fb})$, and where $\lambda\in{\fb}$ satisfies one of the following conditions.
\begin{enumerate}
\item $\lambda\in F^*$,
\item $\lambda=c\sqrt{d}$ for some $c\in F^*$,
\item $\lambda=\sqrt{d'}$ for some $d'\in F^*\setminus {F^*}^2d$.
\end{enumerate} 
\end{Prp}

We shall call Lie bialgebra structures corresponding to these three cases as \emph{basic}, \emph{quadratic} and \emph{twisted}, respectively.

\begin{proof} By Theorem \ref{TPS}, we know that the statement holds either with $\lambda=\alpha+\beta\sqrt{d}$ or $\lambda\notin K$ and $\lambda^2=\alpha+\beta\sqrt{d}$,  where $\alpha,\beta\in F$. In either case,
\begin{equation}\label{r21}
r+r^{21}=\lambda\Omega. 
\end{equation}

Let us first consider the case where $\lambda=\alpha+\beta\sqrt{d}\in K$. Then \eqref{rstar} applies with $\mu\in\fb$, and adding these two equations and applying \eqref{r21} we get
\[\lambda-\overline{\lambda}=2\mu,\]
whence $\mu=\beta\sqrt{d}$. On the other hand, by Lemma \ref{LK}, either $\mu=0$, implying $\lambda=\alpha$, or $\mu=\lambda$, implying $\lambda=\beta\sqrt{d}$. Thus if \eqref{r21} holds with $\lambda\in K$, then $\lambda$ satisfies item (1) or (2).

Consider next the case where $\lambda=\sqrt{d'}$ with $d'=\alpha+\beta\sqrt{d}$. Then $\lambda$ is an element of the splitting field $K'$ of the polynomial $(X^2-\alpha)^2-\beta^2d\in F[X]$. The element of the Galois group $\Gal(K'/F)$ defined by $\sqrt{d}\mapsto-\sqrt{d}$ sends $\lambda$ to $\widehat\lambda=\sqrt{\alpha-\beta\sqrt{d}}$. As in the previous case we add the equations in \eqref{rstar} and apply \eqref{r21}, obtaining
\[\lambda-\widehat{\lambda}=2\mu,\]
and apply Lemma \ref{LK}. If $\mu=0$, then $\widehat\lambda=\lambda$, and upon squaring one gets $\beta=0$ and $d_1=\alpha$. If $\mu=\lambda$, then $\widehat\lambda=-\lambda$, then squaring again gives $\beta=0$. (However this then implies that $\alpha=0$ as well, which is impossible.) Thus if $\lambda=\sqrt{d'}$, then $d'\in F^*$, and item (3) applies.
\end{proof}

\subsection{Quadratic DJ-Lie bialgebra structures}

We shall consider at some length Lie bialgebra structures $\delta$ on $\g$ such that $\delta_{\fb}=\partial r$ with $r=\lambda (\Ad_X\otimes\Ad_X)(r_{DJ})$ with $\lambda$ satisfying item (2) above. We call any $\delta$ satisfying these conditions a \emph{quadratic DJ-Lie bialgebra}. Note that we may assume that in fact $r=\sqrt{d}r_{DJ}$, since non-zero scalar multiples give equivalent Lie bialgebra structures.

We begin by giving a construction of $r=\sqrt{d}r_{DJ}$ using a Manin triple in a convenient way. For the next lemma, we write $\g_+$ for the image of the embedding $\g\to\g_K$, $x\mapsto 1\otimes x$. Moreover we set
\[\g_-=\{A\in\g_K: \forall i: A_{ii}\in F \wedge (j<i\Rightarrow A_{ij}=0)\}.\]
Finally, we define the $F$-bilinear form $\langle,\rangle$ on $\g_K$ by 
\[\langle A+\sqrt{d}B,C+\sqrt{d}D\rangle=2n\tr(AD+BC)\]
for all $A,B,C,D\in\g_+$. Note that if $\kappa$ denotes the Killing form on $\g_K$, then $\langle A+\sqrt{d}B,C+\sqrt{d}D\rangle=\beta$, where $\kappa(A+\sqrt{d}B,C+\sqrt{d}D)=\alpha+\beta\sqrt{d}$ with $\alpha,\beta\in F$.

\begin{Lma} The triple $(\g_K, \g_+, \g_-)$ is a Manin triple with respect to the form $\langle,\rangle$. In particular, the Drinfeld double of $\g$ is isomorphic over $F$ to $\g_K$.
\end{Lma}

We recall that a triple of Lie algebras $(\g',\g_+',\g_-')$ is a \emph{Manin triple} with respect to a non-degenerate bilinear form $b$ on $\g'$ if $\g_+'$ and $\g_-'$ are $b$-isotropic Lie subalgebras of $\g'$ with $\g'=\g_+'\oplus\g_-'$ as a vector space.

\begin{proof} It is clear that $\g_+$ and $\g_-$ are $F$-Lie subalgebras of $\g_K$ having trivial intersection, and counting dimensions one verifies that $\g_K=\g_+\oplus\g_-$ as a vector space. It remains to be shown that $\langle,\rangle$ is non-degenerate and that $\g_+$ and $\g_-$ are isotropic. Let $X\in\g_K$. By Cartan's criterion there exists $Y\in \g_K$ such that $\kappa(X,Y)\neq 0$, and then either $\langle X,Y\rangle\neq 0$ or $\langle X,\sqrt{d} Y\rangle\neq 0$, whence the form is non-degenerate. On the other hand if $X\in \g_+\cup \g_-$, then each diagonal entry of $X^2$, and hence the trace of $X^2$, is in $F$, whence $\langle X,X\rangle= 0$, showing that $\g_+$ and $\g_-$ are isotropic.
\end{proof}

Given a Lie bialgebra structure $\delta$ on a Lie algebra $\g_+$, the Drinfeld double \linebreak $D=\mathcal D(\g_+,\delta)$ is, as a vector space, equal to $g_+\oplus\g_+^*$. The triple $(D,\g_+,\g_+^*)$ is a Manin triple, where the bilinear form is the usual duality pairing, extended to $D$ by being isotropic on $\g_+$ and $\g_+^*$. Conversely, given a Manin triple $(\g',\g_+',\g_-')$, one can construct a Lie bialgebra structure on $\g_+$. This gives a well-known one-to-one correspondence between Manin triples and Lie bialgebra structures.

The above lemma provides us with a Lie bialgebra structure on $\g_+$, as follows. Fix the Cartan subalgebra $\h_K\subset\g_K$ of diagonal matrices with orthogonal basis $\{h_i:1\leq i<n\}$, where
\[h_i=\sum_{j=1}^i E_{jj}-iE_{(i+1)(i+1)}\]
 and let $\Delta$ be the corresponding root system. Let $\{\alpha_1,\ldots,\alpha_m\}$ be the set of positive roots and write $e_{\alpha_j}$ and $e_{-\alpha_j}$, respectively, for the upper triangular and lower triangular $K$-basis element of $\g_K$ corresponding to each $\alpha_j$. Then an $F$-basis for $\g_+$ is given by 
\[B_+=\left((\sqrt{d}h_i)_{i=1}^{n-1}, (e_{\alpha_j}-e_{-\alpha_j})_{j=1}^m,(\sqrt{d}(e_{\alpha_j}+e_{-\alpha_j}))_{j=1}^m\right),\]
and an $F$-basis of $\g_-$, dual to the above with respect to $\langle,\rangle$, is given by
\[B_-=\frac{1}{2n}\left((\frac{1}{i+i^2}h_i)_{i=1}^{n-1}, (-\sqrt{d}e_{\alpha_j})_{j=1}^m,(e_{\alpha_j})_{j=1}^m\right).\]
Setting $r=\sum_{e\in B_+} e\otimes e'$, where for each $e\in B_+$, $e'\in B_-$ is the unique vector not orthogonal to $e$, we find that $r=\sqrt{d}r_{DJ}$.

The following result is well-known and describes the centralizer of $r$.

\begin{Prp} The centralizer $C(r_{DJ})$ of $r_{DJ}$, i.e.\ the set of all $M\in\GL(n,\overline{F})$ such that $(\Ad_M\otimes\Ad_M)(r_{DJ})=r_{DJ}$, consists of all diagonal matrices. 
\end{Prp}

As a consequence, the same holds for $r=\sqrt{d}r_{DJ}$.

\section{Belavin--Drinfeld Cohomology}
We are thus interested in classifying, up to gauge equivalence over $F$, those Lie bialgebra structures $\delta$ on $\g$ that, upon extending scalars to $\overline{F}$, become gauge equivalent to $\sqrt{d}r_{DJ}$. Our tool will be a descent-type argument using certain cohomology sets which we shall introduce. For notational convenience, we continue to write $r$ for $\sqrt{d}r_{DJ}$. The following result implies that we need only extend scalars to $K$ for any two such bialgebra structures to become isomorphic.

\begin{Lma} Let $X\in\GL(n,\overline{F})$, and assume that $(\Ad_X\otimes\Ad_X)(r)$ defines a Lie bialgebra structure on $\g_K$. Then 
\[(\Ad_X\otimes\Ad_X)(r)=(\Ad_Y\otimes\Ad_Y)(r)\]
for some  $Y\in\GL(n,K)$.
\end{Lma}
 
The condition that $(\Ad_X\otimes\Ad_X)(r)$ defines a Lie bialgebra structure on $\g_K$ is clearly necessary for it to define a Lie bialgebra structure on $\g\simeq \g_+\subset \g_K$.

\begin{proof} If $(\Ad_X\otimes\Ad_X)(r)$ defines a Lie bialgebra structure on $\g_K$, then for any $\sigma\in\Gal(\overline{F}/K)$,
\[(\sigma\otimes\sigma)((\Ad_X\otimes\Ad_X)(r))=(\Ad_X\otimes\Ad_X)(r)+\alpha\Omega\]
for some $\alpha\in{\fb}$. A standard computation then gives $\alpha=0$.
Since $(\sigma\otimes\sigma)(r)=r$, the left hand side equals $(\Ad_{\sigma(X)}\otimes\Ad_{\sigma(X)})(r)$, and altogether we get $X^{-1}\sigma(X)\in C(r)$. Since $r=\sqrt{d}r_{DJ}$ and $\sigma$ fixes $\sqrt{d}$, it follows that $X^{-1}\sigma(X)\in C(r_{DJ})$, whence $X^{-1}\sigma(X)$ is diagonal. By an argument similar to that of Lemma 2 of \cite{4046} we conclude that $XD\in\GL(n,K)$ for some diagonal $D\in\GL(n,\overline{F})$. The claim follows as the $r$-matrices $(\Ad_X\otimes\Ad_X)(r)$ and $(\Ad_{XD}\otimes\Ad_{XD})(r)$ are equal.
\end{proof}

Our next concern is thus to determine precisely which $X\in\GL(n,K)$ give rise to $r$-matrices which define Lie bialgebra structures on $\g$. This is the content of the following result.

\begin{Thm}\label{T1} Let $X\in\GL(n,K)$. Then $(\Ad_X\otimes\Ad_X)(r)$ defines a Lie bialgebra structure on $\g$ if and only if $X^*X=D$ for some diagonal $D\in \GL(n,F)$.
\end{Thm}

\begin{proof} By Remark \ref{Rlambda}, an $r$-matrix $s\in\g_K\otimes\g_K$ defines a Lie bialgebra structure on $\g$ precisely when
\[s-s^*=\alpha\Omega\]
for some $\alpha\in K$. Thereby
\[s_{21}-s_{21}^*=\alpha\Omega\]
holds as well. (We write the superscript 21 as a subscript to enhance legibility.) If now $s=(\Ad_X\otimes\Ad_X)(r)$ for some $X\in \GL(n,K)$, then using $r_{DJ}+r_{DJ}^{21}=\Omega$, we find
\[s-s^*+s_{21}-s_{21}^*=2\sqrt{d}\Omega,\]
which combined with the above gives $\alpha=\sqrt{d}$. Now since $r_{DJ}^*=r_{DJ}^{21}=\Omega-r_{DJ}$, and moreover $\Ad_Y\otimes\Ad_Y$ fixes $\Omega$ for any $Y\in\GL(n,K)$, we get
\[\begin{array}{crcl}
&s&=&\sqrt{d}\Omega+s^*\\
\Longleftrightarrow & (X\otimes X)r_{DJ}(X^{-1}\otimes X^{-1})&=&({X^*}^{-1}\otimes {X^*}^{-1})r_{DJ}(X^{-1}\otimes X^{-1})\\
\Longleftrightarrow & (\Ad_{X^*X}\otimes\Ad_{X^*X})(r_{DJ})&=&r_{DJ},\\
\end{array}\]
which is equivalent to $X^*X=D$ for some diagonal matrix $D\in\GL(n,K)$, since the centralizer $C(r_{DJ})$ of $r_{DJ}$ contains no non-diagonal elements. Any such $D$ satisfies $D\in\GL(n,F)$, since for each $1\leq i\leq n$
\[(X^*X)_{ii}=\sum_k N(X_{ki}).\]
This concludes the proof.
\end{proof}

This in fact re-proves that $r$ itself induces a Lie bialgebra structure on $\g$. The above result paves the road for introducing the following tool for the problem of classifying Lie bialgebra structures on $\g$.

\begin{Def} We say that $X\in \GL(n,K)$ is a \emph{diagonal type Belavin--Drinfeld cocycle}, or \emph{cocycle} for short, if $X^*X=D_X$ for some diagonal $D_X\in \GL(n,F)$. The set of cocycles is denoted by $Z_d(r, F, d)$. Two cocycles $X$ and $Y$ are \emph{cohomologous} if $Y=QXD$ for some  $Q \in\U(n,d)(F)$ and diagonal $D\in \GL(n,K)$. The set of cohomology classes is denoted by $H_d^1(r_{DJ},F,d)$. If $c\in H_d^1(r_{DJ},F,d)$ is a cohomology class and $D\in \GL(n,F)$ is diagonal and satisfies $D=D_X$ for some $X\in c$, we say that $D$ \emph{represents} $c$.
\end{Def}

The assignment $X\mapsto \partial((\Ad_X\otimes\Ad_X)(r))$ thus defines a map $\mathcal F$ from $Z_d(r, F, d)$ to the class of all quadratic DJ-Lie bialgebra structures on $\g$. From the definition of equivalence and the description of the centralizer of $r$ it follows that two cocycles $X,Y\in Z_d(r, F, d)$ are cohomologous if and only if $\mathcal F(X)$ and $\mathcal F(Y)$ are gauge equivalent. (Note that if $\mathcal F(X)$ and $\mathcal F(Y)$ are gauge equivalent, with $Y=QXD$ as above, then necessarily $D\in\GL(n,K)$.) This proves the following.

\begin{Prp} There is a one-to-one correspondence between $H_d^1(r_{DJ},F,d)$ and gauge equivalence classes of quadratic DJ-Lie bialgebra structures on $\g$.
\end{Prp}

The cohomology condition can be further simplified as follows.

\begin{Lma}\label{L: iso} Two cocycles $X$ and $Y$ are cohomologous if and only if $D_X=\overline{D}DD_Y$ for some diagonal $D\in\GL(n,K)$. In particular, each diagonal $D\in GL(n,F)$ represents at most one cohomology class.
\end{Lma}

\begin{proof}
If two cocycles $X$ and $Y$ are cohomologous, say with $Y=QXD$ with $Q$ and $D$ as above, then
\[D_Y=Y^*Y=D^*X^*Q^*QXD=\overline{D}DD_X.\]
(This is in fact equivalent to $(D_X)_{ii}$ belonging to the same coset as $(D_Y)_{ii}$ in $F^*/N(K^*)$ for each $1\leq i\leq n$.) Conversely, if $D_Y=\overline{D}DD_X$, then $Q=YD^{-1}X^{-1}$ satisfies $Q^*Q=1$.
\end{proof}

Note that in general, not every diagonal matrix $D_0$ represents a cohomology class. For example if $F=\mathbb R$ and $d=-1$ (whence $\g$ is the ordinary $\su_n$-algebra), then every diagonal element of $X^*X$ is a sum of Euclidean norms, whence no diagonal matrix $D_0$ with negative entries satisfies $D_0=D_X$.

On the other hand, if the field $F$ is such that we can determine the class $\mathcal C(F)$ of all $D_0$ that represent a cohomology class, then, loosely speaking, we obtain a classification upon factoring $\mathcal C(F)$ by copies of $N(K^*)$ in a suitable fashion. The main task is thus to determine, for each $D_0$, whether $D_0=D_X$ for some cocycle $X$. We will fulfill this task over a large class of fields to be specified later. The main step is a proving a sufficient condition on $D_0$ over any field of characteristic zero, in terms of quaternion algebras.

\section{Quaternion Algebras}

Let $a,b\in F^*$. The \emph{quaternion algebra} $(a,b)_F$ is the four-dimensional unital associative $F$-algebra with basis $\{1,i,j,ij\}$ and multiplication given by $1$ being the unity and by
\[\begin{array}{llll}
   i^2=a, & j^2=b, & \text{and} & ji=-ij.
  \end{array}
\]
(As before, we assume that $\ch F=0$. The definition, however, works for any field of characteristic different from 2.) The prototypical example is Hamilton's real division quaternion algebra $(-1,-1)_\mathbb R$ discovered in 1843. Quaternion algebras over arbitrary fields have been extensively studied since, and we shall only recall those results which will be needed for our purposes. A thorough account can be found in e.g.\ \cite{GSz}. The algebra $A=(a,b)_F$ is equipped with a multiplicative quadratic form $Q$, known as the \emph{norm form} of $A$, given in the basis above by
\[Q(x1+yi+zj+wij)=(x^2-ay^2)-b(z^2-aw^2).\]
This is indeed a $2$-Pfister form (see Section 5.2), and there is in fact a one-to-one correspondence between quaternion algebras and $2$-Pfister forms over $F$.

The linear span of $\{1,i\}$ is a commutative subalgebra of $(a,b)_F$. If $a\notin F^{*2}$, then this subalgebra is a quadratic field extension of $F$ isomorphic to $F(\sqrt{a})$. The restriction of the norm form to this field extension is the usual field norm. In general, given $a\in F^*\setminus F^{*2}$, a quaternion algebra $A$ over $F$ contains $F(\sqrt{a})$ as a subalgebra if and only if $A\simeq(a,b)_F$ for some $b\in F^*$. 

Every quaternion algebra $A$ is central simple (i.e.\ the centre of $A$ is $F1$ and $A$ has no proper non-trivial two-sided ideals). Thus $A$ defines an element $[A]$ in the \emph{Brauer group} $\Br(F)$, whose elements are all Brauer equivalence classes of central simple $F$-algebras, with multiplication induced by the tensor product over $F$. Recall that any central simple $F$-algebra $A$ becomes isomorphic to the matrix algebra $M_n(\overline{F})$ for some $n=n(A)\in\mathbb N$ upon extending scalars to an algebraic closure. Here, two central simple algebras $A$ and $B$ are called \emph{Brauer equivalent} if $A\otimes_F M_m(F)\simeq A\otimes_F M_{m'}(F)$ for some $m,m'\in\mathbb N$.
\subsection{Belavin--Drinfeld Cohomology and Quaternion Algebras}
Our first use of quaternion algebras will be to elucidate the structure of the cohomology introduced above. For each $n>0$, let $\mathcal Q^n(F,d)$ be the class of all $n$-tuples of quaternion algebras over $F$ that contain a subalgebra isomorphic to $K=F(\sqrt{d})$. Consider the map $\mathcal G: Z_d(r_{DJ},F,d)\to \mathcal Q(F,d)^n$ defined by mapping the cocycle $X$ to the $n$-tuple $((d,d_1),\ldots,(d,d_n))$, where $\diag(d_1,\ldots,d_n)=D_X$. The following result shows that cohomology classes correspond to isomorphism classes of quaternion algebras.

\begin{Prp}\label{T: q} Let $X,Y\in Z_d(r_{DJ},F,d)$. Then $X$ and $Y$ are cohomologous if and only if $\mathcal G(X)\simeq\mathcal G(Y)$.
\end{Prp}

Here, two $n$-tuples $(A_1,\ldots,A_n)$ and $(A_1',\ldots,A_n')$ of quaternion algebras are said to be isomorphic if for all $1\leq i\leq n$, $A_i\simeq A_i'$. 

\begin{proof} By Lemma \ref{L: iso}, $X$ and $Y$ are cohomologous if and only if $D_Y=\overline{D}DD_X$ for some diagonal $D\in\GL(n,K)$. If $D_X=\diag(d_1,\ldots,d_n)$ and $D_Y=\diag(d_1',\ldots,d_n')$, this is equivalent to the statement that
\[\forall i\in \{1,\ldots,n\}, d_i'd_i^{-1}\in N(K^*),\]
since each diagonal element of $\overline{D}D$ is of the form $\overline{p}p=N(p)$ for some $p\in K^*$. Now it is known that two quaternion algebras $(a,b)_F$ and $(a,b')_F$ are isomorphic if and only if $b'b^{-1}\in N(F(\sqrt{a})^*)$. Thus $X$ and $Y$ are cohomologous if and only if their images under $\mathcal G$ are component-wise isomorphic, as desired. 
\end{proof}

\subsection{Existence of Cocycles and Nested Quaternion Algebras}
The map $\mathcal G$ is in general not surjective since, as remarked earlier, not every diagonal matrix arises as $D_X$ for some cocycle $X$. We will therefore give a sufficient condition for a diagonal matrix to satisfy this, in terms of norms of quaternion algebras. While we do not know of a necessary condition over general field, the condition we will give will be enough to obtain a classification over an important class of fields.

Let $n$ be a positive integer and $\Delta=\{d_1,\ldots,d_n\}\subset F^*$. We call $\Delta$ \emph{norm closed} or, briefly, \emph{closed} if $\prod_{i\in \Delta} d_i\in N(K)$, and \emph{quaternionically nested} if
\[\exists \sigma\in S_{n}: \forall i\in\{1,\ldots,n\}: d_{\sigma(i)} \text{\ is the norm of some\ } q\in\left(d,\prod_{k=0}^{i-1} d_{\sigma(k)}\right)_F,\]
where $S_{n}$ is the symmetric group on $\{1,\ldots,n\}$, and where we set $d_0=-1$ and extend $\sigma$ to $\{0,\ldots,n\}$ by $\sigma(0)=0$. By definition of the norm of a quaternion algebra, the property of being quaternionically nested amounts to saying that for each $i\in\{1,\ldots,n\}$ there exist $x,y\in K$ such that
\[N(x)+N(y)\prod_{k=1}^{i-1} d_{\sigma(k)}=d_{\sigma(i)}.\]

We are interested in finite subsets $I\subseteq F^*$ such that
\begin{equation}\tag{Q}\label{Q}
\text{every norm closed subset of\ } I \text{\ is quaternionically nested}.
\end{equation}
Note that this property trivially holds for all $I$ over fields where the norm of every quaternion algebra is surjective, as is the case over $p$-adic fields and, more generally, fields of cohomological dimension at most 2. We can now prove the following.

\begin{Thm}\label{T2} Assume that $D=\diag(d_1,\ldots,d_n)\in\GL(n,F)$ where $\{d_1,\ldots,d_n\}$ is norm closed and satisfies property \emph{(Q)}. Then there exists a matrix $X\in \GL(n,K)$ satisfying $X^*X=D$. 
 \end{Thm}

If two finite sets $I$ and $J$ of $F^*$ are norm closed with $I\subseteq J$, then $J\setminus I$ is closed. Iterating this process, one can partition $\{d_1,\ldots,d_n\}$ into closed subsets, none of which contains a proper, non-empty closed subset. If for each such set $I=\{d_{i_1},\ldots,d_{i_m}\}$ we can construct $X_I$ such that $X_I^*X_I=\diag(d_{i_1},\ldots,d_{i_m})$, then the block-diagonal matrix $X$ with blocks $X_I$ will satisfy $X^*X=D$.

\begin{proof} In view of the above remark, we may assume that $\{d_1,\ldots,d_n\}$ contains no proper, non-empty closed subsets. We may also assume, upon renumbering the diagonal elements of $D$, that $\sigma$ is the identity permutation. 

We will construct the rows $x_i$ of $X^*$ inductively. Note that $X^*X=D$ is equivalent to $(x_i,x_j)=\delta_{ij}d_i$, where the $F$-bilinear pairing $(,):K^n\times K^n\to K$ is defined by 
\[\left(\left(x_i^{(1)},\ldots,x_i^{(n)}\right),\left(x_j^{(1)},\ldots,x_j^{(n)}\right)\right)=\sum_k x_i^{(k)}\overline{x_j^{(k)}}.\]
In particular 
\[(x_i,x_i)=\sum_k N\left(x_i^{(k)}\right).\]

If $n=1$, then $d_1=N(a_1)$ for some $a_1\in K$, and setting $X=a_1$, we are done.

If $n>1$, then we set
\[x_1=(a_1,a_2,0,\ldots, 0)\]
where $a_1,a_2\in K$ satisfy $N(a_1)+N(a_2)=d_1$. Such $a_1$ and $a_2$ exist by property (Q), and then $(x_1,x_1)=d_1$. If $n=2$, then $d_2=N(\mu_1)d_1$, and setting
\[x_2=\mu_1(-\overline{a_1},\overline{a_2}),\]
we are done. If $n>2$, then by property (Q) there exist $\mu_1,a_3\in K$ such that $N(\mu_1)d_1+N(a_3)=d_2$, and then we set
\[x_2=(-\mu_1\overline{a_1},\mu_1\overline{a_2},a_3,0,\ldots, 0).\]
In both cases $(x_1,x_2)=0$ and $(x_2,x_2)=d_2$. Note that in the latter case, \[N(x_2^{(1)})+N(x_2^{(2)})=N(\mu_1)d_1,\] 
and $a_3\neq 0$ since $\{d_1,d_2\}$ is not closed. 

Assume that 
\[x_i=\left(x_i^{(1)},\ldots,x_i^{(i+1)},0,\ldots,0\right)\]
has been constructed, $2\leq i<n-1$, and satisfies $(x_i,x_j)=\delta_{ij}d_j$ for all $j<i$ as well as $x_i^{(i+1)}\neq 0$ and
\[\exists \lambda_i\in K: \sum_{j=1}^i N\left(x_i^{(j)}\right)=N(\lambda_i)\prod_{j=1}^{i-1}d_j.\]
Denoting $\sum_{j=1}^i N\left(x_i^{(j)}\right)$ by $s_i$, we then set
\[x_{i+1}=\left(\mu_ix_i^{(1)},\ldots,\mu_ix_i^{(i)},-\mu_is_i/\overline{x_i^{(i+1)}},a_{i+2},0,\ldots,0\right)\]
where $\mu_i,a_{i+2}$ satisfy
\[\frac{s_id_i}{N\left(x_i^{(i+1)}\right)}N(\mu_i)+N(a_{i+2})=d_{i+1}.\]
Again, such elements exist by property (Q), since
\[\frac{s_id_i}{N\left(x_i^{(i+1)}\right)}=N\left(\lambda_i/x_i^{(i+1)}\right)\prod_{j=1}^{i}d_j.\]
Thus constructed, $x_{i+1}$ satisfies the properties assumed for $x_i$. Indeed, we have $x_{i+1}^{(i+2)}=a_{i+2}\neq 0$ by the non-closedness of $\{d_1,\ldots,d_{i+1}\}$, and \[(x_j,x_{i+1})=\mu_i(x_j,x_i)=0\]
whenever $j<i$. An easy computation shows that $(x_i,x_{i+1})=0$. Furthermore,
\[\begin{array}{rcl}\sum_{j=1}^{i+1}N\left(x_{i+1}^{(j)}\right)&=&s_i+\frac{s_i^2}{N\left(x_i^{(i+1)}\right)}\\
  &=&\frac{s_i}{N\left(x_i^{(i+1)}\right)}\left(N\left(x_i^{(i+1)}\right)+s_i\right)\\
  &=&s_id_i/N\left(x_i^{(i+1)}\right)=N\left(\lambda_i/x_i^{(i+1)}\right)\prod_{j=1}^id_k.
  \end{array}\]
This finally implies  
\[\begin{array}{rcl}(x_{i+1},x_{i+1})&=&N(\mu_i)\left(s_i+\frac{s_i^2}{N\left(x_i^{(i+1)}\right)}\right)+N(a_{i+2})\\
  &=&N\left(\mu_i/x_i^{(i+1)}\right)s_id_i+N(a_{i+2})=d_{i+1}.
  \end{array}\]
In this fashion we construct $x_3,\ldots,x_{n-1}$ inductively. Having done so it remains to construct $x_n$, for which we set
\[y=\left(x_{n-1}^{(1)},\ldots,x_{n-1}^{(n-1)},-s_{n-1}/\overline{x_{n-1}^{(n)}}\right).\]
Then as before we have $(x_j,y)=0$ for all $j<n$, and a computation similar to the above shows that
\[(y,y)=N\left(\lambda_{n-1}/x_{n-1}^{(n)}\right)\prod_{j=1}^{n-1}d_k\]
and since $\{d_1,\ldots,d_n\}$ is closed there exists $\mu_n$ such that $x_n=\mu_ny$ satisfies
\[(x_n,x_n)=d_n.\]
This completes the proof.
\end{proof}

\begin{Rk}\label{Rk} Observe that the requirement that $\{d_1,\ldots,d_n\}$ be closed is necessary, for indeed, if $X^*X=D$, then
\[\prod_{i=1}^nd_i=\det(D)=\det(X^*)\det(X)=\overline{\det(X)}\det(X)=N(\det(X)).\]
\end{Rk}

\section{Classification over Special Fields}
\subsection{Fields of Cohomological Dimension at Most 2}
Recall that $F$ is assumed to have characteristic zero. In this section, we in addition assume that $F$ has the property that the norm of any quaternion algebra over $F$ is surjective. 

It is known (see e.g.\ \cite{Se}) that this property is satisfied by any field of cohomological dimension at most 2; these include all algebraically closed fields, $p$-adic fields (for any prime $p$), totally imaginary number fields, function fields of surfaces and curves over algebraically closed fields, and Merkurjev's Tower of fields. Such fields appear in the literature, for example, in connection with Serre's Conjecture II on the vanishing of Galois cohomology. (See e.g.\ \cite{G} for a survey.) From another viewpoint, let $d$ and $i$ be positive integers. A field $F$ is said to have the property $C_i^{(d)}$ if every homogeneous polynomial of degree $d$ in $n>d^i$ variables has a non-trivial zero. If $F$ satisfies this for any $d>0$, then $F$ is said to have the property $C_i$. It is an established fact that if a quadratic form $q$ in $n$ variables over a field $F$ of characteristic 0 represents zero, then there exist $a_1,\ldots,a_n\in F^*$ such that $q(a_1,\ldots,a_n)=0$. 
Therefore if $F$ has the 
property $C_2^{(2)}$, and in 
particular if $F$ is a $C_2$-field, then the norm of any quaternion algebra over $F$ is surjective.

Under this assumption, the property (Q) assumed in Theorem \ref{T2} above is trivially satisfied. Thus Lemma \ref{L: iso} and Theorem \ref{T2} then imply the following.

\begin{Prp} Let $F$ be a field over which the norm of any quaternion algebra is surjective, $d\in F^*\setminus {F^*}^2$, and set $K=F(\sqrt{d})$. Then there is a one-to-one correspondence between $H_d^1(r_{DJ},F,d)$ and $(F^*/N(K^*))^{n-1}$.
\end{Prp}

\begin{proof} Define the map $Z_d^1(r_{DJ},F,d)\to(F^*/N(K^*))^{n-1}$ by 
\[X\mapsto (d_1 N(K^*),\ldots,d_{n-1}N(K^*)),\]
where $D_X=\diag(d_1,\ldots,d_n)$. This map is surjective by Theorem \ref{T2} since for any $d_1,\ldots,d_{n-1}\in F^*$ there exists $d_n\in F^*$ such that $\{d_1,\ldots,d_n\}$ is norm closed. It induces a map $H_d^1(r_{DJ},F,d)\to(F^*/N(K^*))^{n-1}$, which is well-defined by Lemma \ref{L: iso}, and injective by the same lemma, since the requirement that $\det D_X\in N(K^*)$ for any cocycle $X$ implies that the class of $d_n$ in $F^*/N(K^*)$ is determined by those of $d_1,\ldots,d_{n-1}$. 
\end{proof}

In terms of the Brauer group, the following holds, where by $S(F)$ we denote a transversal of the cosets of $F^{*2}$ in $F^*$.

\begin{Cor} Let $F$ be a field over which the norm of any quaternion algebra is surjective. Let $H_d^1(F)=\bigsqcup_{d\in S(F)}H_d^1(r_{DJ},F,d)$. Then there is a map $H_d^1(F)\to\Br(F)^{n-1}$, whose image generates $\Br_2(F)^{n-1}$, where $\Br_2(F)$ is the subgroup of $\Br(F)$ formed by all 2-torsion elements.
\end{Cor}

\begin{proof} The existence of such a map follows from Proposition \ref{T: q}, the map being induced by $\mathcal G$ defined there. An argument similar to that of the above proof shows that its image consists of all $n-1$-tuples of Brauer classes of quaternion algebras. Due to the celebrated theorem by Merkurjev \cite{Me}, these classes generate $\Br_2(F)$.
\end{proof}

\subsection{Extensions of Real Closed Fields}
A field $R$ is said to be \emph{formally real} if $-1$ cannot be represented as a sum of squares in $R$. A formally real field with no formally real algebraic field extension $R\subsetneq R'$ is called \emph{real closed}. The real field $\mathbb R$ is real closed, and in fact every real closed field is elementarily equivalent to $\mathbb R$ in the language of rings, meaning that it satisfies precisely the same first-order statements in this language. We will now extend the results obtained above to the fields $R$, $R(X)$ and $R(X,Y)$, where $R$ is real closed. The Artin--Schreier Theorem asserts that if a field is real closed, then it admits a unique ordering: namely, $a\leq b$ if and only if $b-a$ is a square in $R$. We will call a rational function $f\in R(X_1,\ldots,X_n)$ \emph{positive} if $f(x_1,\ldots,x_n)\geq 0$ for all $(x_1,\ldots,x_n)\in R^n$ at which $f$ is defined.

Hilbert asked, in what became known as Hilbert's Seventeenth Problem, whether it is true that every positive $f\in\mathbb R(X_1,\ldots,X_n)$ is a finite sum of squares. This was answered in the affirmative by Artin, without giving any bound for the number of squares needed. (The corresponding question for \emph{polynomials} was answered in the negative by Hilbert himself.) Pfister \cite{PfD} later gave the upper bound $2^n$ for $R(X_1,\ldots,X_n)$ for each $n\in\mathbb N$ and each real closed field $R$. The bound is sharp for $n\leq 2$. Pfister in fact proved the following more general statement in \cite{PfE}.

\begin{Prp}\label{Pf} Let $R$ be a real closed field, $n\in\mathbb N$, and let $f\in R(X_1,\ldots,X_n)$ be positive. Then every $n$-Pfister form over $R(X_1,\ldots,X_n)$ represents $f$.
\end{Prp}

A quadratic form $q$ over a field $F$ is called an \emph{$n$-Pfister form} or a \emph{multiplicative $n$-form} if $q=\langle 1,a_1\rangle\otimes\cdots\otimes \langle 1,a_n\rangle$ for some $a_1,\ldots,a_n\in F^*$, where 
\[\langle 1,a\rangle(x,y)=x^2-ay^2.\] One then writes $q=\langle\langle a_1,\ldots,a_n\rangle\rangle$. A consequence of the above proposition is that the $n$-Pfister form $\langle\langle 1,\ldots,1\rangle\rangle$ represents any $f\in R(X_1,\ldots,X_n)$, which is precisely the statement that $f$ is a sum of $2^n$ squares.

\begin{Rk} A quadratic form is a $2$-Pfister form if and only if it is the norm of a quaternion algebra. More specifically, $\langle 1,a\rangle\otimes\langle 1,b\rangle$ is the norm of $(a,b)_F$ for any $a,b\in F^*$. This is true over any field of characteristic not 2. 
\end{Rk}

Let now $F$ be any of $R$, $R(X)$ or $R(X,Y)$ with $R$ real closed. An extension $K=F(\sqrt{d})$ with $d\in F^*$ is called \emph{imaginary} if $d$ negative (i.e.\ $-d$ positive in $F$). This terminology is in analogy with the classical notion of an imaginary number field, i.e.\ a quadratic extension $\mathbb Q(\sqrt{d})$ of $\mathbb Q$ with $d<0$. In view of Pfister's results, we can prove the following analogue of Theorem \ref{T2}.

\begin{Prp}\label{P3} Let $F$ be any of $R$, $R(X)$ or $R(X,Y)$ with $R$ a real closed field and $K=F(\sqrt{d})$ an imaginary extension, and let $D=\diag(d_1,\ldots,d_n)\in\GL(n,F)$. The following are equivalent.  
\begin{enumerate}
 \item The set $\{d_1,\ldots,d_n\}$ is norm closed and for each $1\leq i\leq n$, $d_i$ is positive.
 \item There exists a matrix $X\in \GL(n,K)$ satisfying $X^*X=D$
\end{enumerate} 
\end{Prp}

\begin{proof} If $D$ satisfies (1), then $\{d_1,\ldots,d_n\}$ satisfy property (Q) since any positive element in $F$ is in the image of the norm of any quaternion algebra over $F$, by Proposition \ref{Pf} and the subsequent remark. Then Theorem \ref{T2} applies and implies (2). To prove the converse, note, as in Remark \ref{Rk}, that if $X^*X=D$, then $\det D=N(\det X)$. Moreover,  
\[d_i=\sum_jN((X)_{ji})\]
for each $i$. From this follows that $\prod_{i=1}^nd_i\in N(K)$, and moreover each $d_i$ is positive since the negativity of $d$ implies that $N(x)$ is positive for each $x\in K$. This completes the proof.
\end{proof}

We therefore have the following classification result, where $F^+$ denotes the set of all positive elements of $F^*$ whenever $F=R(X_1,\ldots,X_n)$ with $R$ real closed and $n\leq 2$.

\begin{Prp} Let $F$ be a rational function field in at most two indeterminates over a real closed field, and $d\in F^*$ negative. Then there is a one-to-one correspondence between $H_d^1(r_{DJ},F,d)$ and $(F^+/N(K^*))^{n-1}$
\end{Prp}

\subsection{Norm Classes of Certain Field Extensions}
The above results classify quadratic DJ-Lie bialgebra structures on $\g$  whenever all quaternion algebras over $F$ have surjective norms, up to a description of $F^*/N(K^*)$, and whenever $F$ a rational rational function field in at most two indeterminates over a real closed field and $K$ is imaginary, up to a description of $F^+/N(F(\sqrt{d}))$. As the next result shows, these groups are in some cases trivial.

\begin{Prp} Let $F$ be a field.
\begin{enumerate}
\item If $F$ is a $C_1^{(2)}$-field, and $d\in F^*$, then $N(F(\sqrt{d})^*)=F^*$.
\item If $F$ is a $p$-adic field and $d\in F^*\setminus {F^*}^2$, then $F^*/N(F(\sqrt{d})^*)$ is a cyclic group of order 2.
\item If $F=R$ or $F=R(X)$ with $R$ real closed, and $d\in (-1)F^+$, then $N(F(\sqrt{d}))=F^+$.
\item If $F=\mathbb C(X,Y)$, and $d\in F^*\setminus {F^*}^2$, then $F^*/N(F(\sqrt{d})^*)$ is infinite.\footnote{We are grateful to Professor A.\ Merkurjev for drawing our attention to this example.}
\end{enumerate}
\end{Prp}

Note that the first item subsumes the case $F=C(X)$ with $C$ algebraically closed.

\begin{proof} If $F$ is a $C_1^{(2)}$-field, then every quadratic form in at least 3 variables represents zero non-trivially, whence every quadratic form in two variables represents every $a\in F^*$. Since $N=\langle 1,d\rangle$ is such a form this proves the first item. The second item is classical, while the third follows from Proposition \ref{Pf} since $\langle 1,d\rangle$ is a $1$-Pfister form. As for the forth item, we may assume that $d\in \mathbb C[X,Y]$. Let $p$ be a prime factor of $d$, and let $E$ be the field of fractions of $\mathbb{C}[X,Y]/(p)$. Then $E$ is a finite field extension of $\mathbb{C}(X)$, and the quotient projection onto $\mathbb{C}[X,Y]/(p)$ defines a map $F\to E$. A norm element $a\in F$ is mapped to a square in $E$. Since the square class group of $\mathbb C(X)$, and hence of $E$, is infinite, the result follows.
\end{proof}

In this context, one may note that a diagonal element $a$ of $X^*X$ for a cocycle $X$ is a sum of $n$ norms, i.e.\ $a=q(x_1,\ldots,x_n)=N(x_1)+\cdots+N(x_n)$. One may ask if the image of the form $q$ always forms a subgroup of $F^*$. This is equivalent to asking whether the form $q$ is \emph{multiplicative}, i.e.\ for each $x,y\in F^n$ there exists $z\in F^n$ such that $q(x)q(y)=q(z)$. If $n=2^m$ for some $m\in\mathbb N$, this form is a Pfister form. The answer is, in general, no when $n$ is not a power of 2, as detailed by the following result due to Pfister (see \cite{PfD}).

\begin{Prp} Let $q$ be a quadratic form in $n$ variables over an arbitrary field $F$ of characteristic zero. Then $q$ is multiplicative if and only if either $q$ is isotropic, or $n=2^m$ for some $m\in\mathbb N$ and $q$ is an $m$-Pfister form.
\end{Prp}

\begin{Rk} Together with Theorem \ref{T2}, this implies that over any field of characteristic zero with quadratic extension $K=F(\sqrt{d})$, if $\g=\su(2,F,d)$, then the set $H_d^1(r_{DJ},F,d)$ is a multiplicative group (in fact, a subgroup of $F^*/N(K^*)$). Moreover, if $\g=\su(2^m,F,d)$ for some $m\in \mathbb N$, then $H_d^1(r_{DJ},F,d)$ contains, in a natural way, such a subgroup. An interesting question is under which conditions $H_d^1(r_{DJ},F,d)$ is then itself a group.
 
\end{Rk}

\section{The General Non-Twisted Case}
We shall now consider Lie bialgebra structures $\delta$ on $\g$ such that $\delta_{\fb}=\partial r$ with $r=\lambda (\Ad_X\otimes\Ad_X)(r_{BD})$, for more general choices of $\lambda$ and $r_{BD}$. As regards $\lambda$, by Proposition \ref{Plambda} we may assume that $\lambda=1$, $\lambda=\sqrt{d}$ or $\lambda=\sqrt{d'}$ for some $d'\in F$. Recall that we refer to the three types as \emph{basic}, \emph{quadratic} and \emph{twisted}, respectively. In this section we will consider the first two cases, and in the next we consider the last. 

As above we fix the diagonal Cartan subalgebra $\h_K$ of $\g_K$ with the above orthonormal basis $\{h_i\}$, root system $\Delta$ and a subset $\Delta^+$ of positive roots. We write $e_{\alpha}$ for the Chevalley generator corresponding to $\alpha\in\Delta$, $\Omega$ for the Casimir element, and $\Omega_0$ for its Cartan part.

We introduce the anti-diagonal matrix $S\in\GL(n,F)$ defined by $S_{ij}=\delta_{i,n+1-j}$, and the map $s:\Gamma\to\Gamma$ defined by $s(\alpha_i)=\alpha_{n-i}$. 

When $r$ is $\GL(n,\fb)$-equivalent to $\lambda r_{BD}$ with $\lambda\in K$, one may assume the equivalence to be effected by some $Y\in\GL(n,K)$ as follows.

\begin{Prp}\label{inK} Let $r=\lambda (\Ad_X\otimes\Ad_X)(r_{BD})$ with $\lambda\in K$, $X\in\GL(n,\fb)$ and $r_{BD}$ a Belavin--Drinfeld $r$-matrix. If $\partial r$ defines a Lie bialgebra structure on $\sll(n,K)$, then
\begin{enumerate}
\item every $\sigma\in\Gal(\fb/K)$ satisfies $\sigma(r_{BD})=r_{BD}$ and $X^{-1}\sigma(X)\in C(r_{BD})$, and
\item $r=\lambda (\Ad_Y\otimes\Ad_Y)(r_{BD})$ for some $Y\in\GL(n,K)$.
\end{enumerate}
\end{Prp}

A proof of the first item is given in \cite{PS}, while the second is established in \cite{KKPS}.

The following result will be helpful when determining the possible choices of $X\in\GL(n,\fb)$. As in Theorem \ref{TBD} we decompose $r_{BD}$ as
\[r_{BD}=r_0+r_1.\]
Since $r_0+r_0^{21}=\Omega_0$ we have $r_0=r_s+\frac{1}{2}\Omega_0$ with $r_s$ skewsymmetric in the sense that $r_s+r_s^{21}=0$.

\begin{Lma}\label{lcartan} Let $r_{BD}=r_0+r_1$ be a Belavin--Drinfeld $r$-matrix and $Y\in\GL(n\fb)$. If $\sigma\in\Gal(\fb/F)$ satisfies
\[(\Ad_Y\otimes\Ad_Y)(r_0'+r_1)=r_{BD}\]
for some $r_0'\in\h_{\fb}\otimes\h_{\fb}$, then $r_0'=r_0$ and $Y\in C(r_{BD})$.
\end{Lma}

This generalizes Theorem 3 of \cite{4046}, and the proof is analogous; we give a brief sketch of it here.

\begin{proof} We extend scalars to $\fb$ and apply the map $\Phi:\g\otimes\g\to\End(\g)$ defined via $\Phi(a\otimes b)=\kappa(a,u)b$, where $\kappa$ is the killing form. Consider, for each $\lambda \in \fb$ the generalized eigenspace
\[\g^\lambda=\bigcup_{n>0}\Ker(\Phi(r_{BD})-\lambda)^n,\]
and set
\[\g'=\bigoplus_{\lambda\notin \{0,1\}}g^\lambda.\]
One then finds that $\g'+\g^0=\mathfrak b_-$, and $\g'+\g^1=\mathfrak b_+$, the Borel subalgebras of $\g_{\fb}$. Since $r_0'+r_1$ only differs from $r_{BD}$ by a term in $\h_{\fb}\otimes\h_{\fb}$, the same holds for that $r$-matrix. Thus $\Ad_Y$ preserves the Borel subalgebras, whence $Y$ is diagonal and $\Ad_Y$ leaves $\h_{\fb}$ invariant. The result follows upon inspecting the diagonal and off-diagonal parts of the left and right hand sides of the equation.
\end{proof}

\subsection{Non-DJ-Lie bialgebra Structures of Quadratic Type}
One may ask if there exist other Belavin--Drinfeld $r$-matrices $r_{BD}\neq r_{DJ}$ such that $\partial(\sqrt{d} r_{BD})$ induces a Lie bialgebra structure on $\g$. The following result severely limits the possibilities.

\begin{Prp}\label{onlydj} Let $r_{BD}$ be a Belavin--Drinfeld $r$-matrix associated to the admissible triple $(\Gamma_1,\Gamma_2,\tau)$. If $\delta$ is a Lie bialgebra structure on $\g$ such that $\delta_{{\fb}}$ is gauge equivalent to $\partial(\sqrt{d} r_{BD})$ over ${\fb}$, then $\Gamma_1=\Gamma_2=\emptyset$. 
\end{Prp}

The proof mimics a technique used in \cite{4046}.

\begin{proof} Set $\delta_{\fb}=\partial r$ with $r=\sqrt{d}(\Ad_X\otimes\Ad_X)(r_{BD})$. By Proposition \ref{inK} we may assume that $X\in\GL(n,K)$. Then \eqref{rstar} together with the fact that $r+r_{21}=\sqrt{d}\Omega$ and $\overline{\sqrt{d}}=-\sqrt{d}$ imply that
\[(\Ad_{X^*X}\otimes\Ad_{X^*X})(r_{BD})=\Omega-r_{BD}^*,\]
and applying the map $s\mapsto s^{21}$ to both sides gives
\begin{equation}\label{djcartan}
(\Ad_{X^*X}\otimes\Ad_{X^*X})(r_{BD}^{21})=r_{BD}^*, 
\end{equation}
where on the right hand side we have used $r_{BD}+r_{BD}^{21}=\Omega$.
Now, with the same notation as in Theorem \ref{TBD},
\[r_{BD}=r_0+\sum_{\alpha\in\Delta^+} e_{\alpha}\otimes e_{-\alpha}+\sum_{\alpha\in\mathrm{Span}(\Gamma_1)^+}\sum_{k\in\mathbb N} e_{\alpha}\wedge e_{-\tau^k(\alpha)}\]
whence
\[r_{BD}^*=\overline{r_0}+\sum_{\alpha\in\Delta^+} e_{-\alpha}\otimes e_{\alpha}+\sum_{\alpha\in\mathrm{Span}(\Gamma_1)^+}\sum_{k\in\mathbb N} e_{-\alpha}\wedge e_{\tau^k(\alpha)}\]
while 
\[r_{BD}^{21}=r_0^{21}+\sum_{\alpha\in\Delta^+} e_{-\alpha}\otimes e_{\alpha}+\sum_{\alpha\in\mathrm{Span}(\Gamma_1)^+}\sum_{k\in\mathbb N} e_{-\tau^k(\alpha)}\wedge e_{\alpha}.\]
Since $r_{BD}^{21}$ is gauge equivalent to $r_{BD}^*$ this implies that for each $\alpha(\Gamma_1)$ and $k\in\mathbb N$ there is $\alpha'\in(\Gamma_1)$ with $\tau^k(\alpha)=\alpha'$. Since for some $k$ we have $\tau^k(\alpha)\in\Gamma_2\setminus\Gamma_1$ by definition of an admissible triple, this implies that $\Gamma_1=\Gamma_2=\emptyset$ as desired.
 \end{proof}

Thus $r_{BD}=r_s+r_{DJ}$, where $r_s\in\h_K\wedge\h_K$. Thus \eqref{djcartan} implies
\[(\Ad_{X^*X}\otimes\Ad_{X^*X})(-r_s+r_{DJ}^{21})=\overline{r_s}+r_{DJ}^{21},\]
whence applying the mapping $s\mapsto s^{21}$ and rearranging terms we get
\[(\Ad_{(X^*X)^{-1}}\otimes\Ad_{(X^*X)^{-1}})(-\overline{r_s}+r_{DJ})={r_s}+r_{DJ},\]
whence Lemma \ref{lcartan} gives $\overline{r_s}=-r_s$ and $X^*X\in C(r_{DJ})$.  Therefore, we have the following.

\begin{Cor} Assume that a Lie bialgebra structure $\delta$ on $\g$ satisfies $\delta_K=\partial r$ with $r=\sqrt{d}(\Ad_X\otimes\Ad_X)(r_{BD})$ for some $X\in\GL(n,K)$. Then $X^*X=D_X$ for some diagonal $D\in\GL(n,F)$ and $r_{BD}=r_s+r_{DJ}$ for some $r_s\in(\h_K\wedge\h_K)$ with $\overline{r_s}=-r_s$.
\end{Cor}

Note that any quadratic Lie bialgebra structure $\delta$ on $\g$ satisfies $\delta_K=\partial r$ with $r$ as above, and that any such $r$ defines a quadratic Lie bialgebra structure on $\g$.

\subsection{Lie Bialgebra Structures of Basic Type}
We shall now consider the case where $\delta_{\fb}=\partial r$, where $r=\lambda (\Ad_X\otimes\Ad_X)(r_{BD})$ with $\lambda\in F^*$. We may thus assume that $\lambda=1$ and, by Proposition \ref{inK}, that $X\in\GL(n,K)$.

\begin{Prp}\label{pbd} Let $r=(\Ad_X\otimes\Ad_X)(r_{BD})$ for some Belavin--Drinfeld $r$-matrix $r_{BD}=r_0+r_1$ over $K$ and $X\in \GL(n,K)$. Then $\partial r$ defines a Lie bialgebra structure on $\g$ if and only if $X^*X=SD$ for some $D\in C(r_{BD})$, $\overline{r_0}=(\Ad_S\otimes\Ad_S)(r_0)$, and the corresponding admissible triple $(\Gamma_1,\Gamma_2,\tau)$ satisfies $s(\Gamma_i)=\Gamma_i$ for $i\in\{1,2\}$ and $s\tau=\tau s$.
\end{Prp}

\begin{proof} As in the proof of Proposition \ref{Plambda} we find that $r$ defines a Lie bialgebra structure on $\g$ if and only if $r=r^*$, i.e.\
\[(\Ad_{X^*X}\otimes\Ad_{X^*X})(r_{BD})=r_{BD}^*\]
and applying $\Ad_S\otimes\Ad_S$ to both sides this is equivalent to 
\begin{equation}\label{sbd}
(\Ad_{SX^*X}\otimes\Ad_{SX^*X})(r_{BD})=(\Ad_S\otimes\Ad_S)(r_{BD}^*). 
\end{equation}
Now,
\[r_{BD}=r_0+\sum_{\alpha\in\Delta^+} e_{\alpha}\otimes e_{-\alpha}+\sum_{\alpha\in\mathrm{Span}(\Gamma_1)^+}\sum_{k\in\mathbb N} e_{\alpha}\wedge e_{-\tau^k(\alpha)}.\]
Upon noticing that $\Ad_S(e_{\alpha})=e_{-s(\alpha)}$ for any $\alpha\in\Gamma$, and following the proof of Proposition \ref{onlydj}, we get
\[(\Ad_S\otimes\Ad_S)(r_{BD}^*)=r_0'+\sum_{\alpha\in\Delta^+} e_{s(\alpha)}\otimes e_{-s(\alpha)}+\sum_{\alpha\in\mathrm{Span}(\Gamma_1)^+}\sum_{k\in\mathbb N} e_{s(\alpha)}\wedge e_{-s\tau^k(\alpha)},\]
where $r_0'=\Ad_S\otimes\Ad_S(\overline{r_0})\in \h_K\otimes\h_K$.
If, as in \eqref{sbd}, these are gauge equivalent, we may proceed as in the proof of Proposition \ref{onlydj}. We then find that for each $\alpha\in\Gamma_1$ and $k\in\mathbb N$ there is $\alpha'\in\Gamma_1$ with $s(\alpha)=\alpha'$ and $\tau^k(\alpha')=s\tau^k(\alpha)$. This implies that $s\tau=\tau s$ and $s(\Gamma_i)=\Gamma_i$ for $i=1,2$.

Furthermore, if this holds, then \eqref{sbd} becomes
\[(\Ad_{SX^*X}\otimes\Ad_{SX^*X})(r_{BD})=r_0'+r_1.\]
Lemma \ref{lcartan} with $Y=(SX^*X)^{-1}$ then gives $Y\in C(r_{BD})$ and $r_0'=r_0$. If conversely $X$, $r_0$, $r_0'$ and the admissible triple satisfy these conditions, then \eqref{sbd} holds. This completes the proof.
\end{proof}

\begin{Def} Let $r_{BD}$ be a Belavin--Drinfeld $r$-matrix with associated admissible triple $(\Gamma_1,\Gamma_2,\tau)$ satisfying the conditions of Proposition \ref{pbd}. We say that \linebreak $X\in \GL(n,K)$ is an \emph{anti-diagonal type Belavin--Drinfeld cocycle} if $X^*X=SD_X$ for some $D_X\in C(r_{BD})$. The set of cocycles is denoted by $Z_a(r_{BD}, F, d)$. Two cocycles $X$ and $Y$ are \emph{cohomologous} if $Y=QXD$ for some  $Q \in\U(n,d)(F)$ and diagonal $D\in \GL(n,K)$. The set of cohomology classes is denoted by $H_a^1(r_{BD},F,d)$. If $c\in H_a^1(r_{BD},F,d)$ is a cohomology class and $D\in \GL(n,K)$ is diagonal and satisfies $D=D_X$ for some $X\in c$, we say that $D$ \emph{represents} $c$.
\end{Def}

As in the case of diagonal type cocycles, the assignment $X\mapsto \partial(\Ad_X\otimes\Ad_X)(r)$ defines a map $\mathcal F'$ from $Z_a(r_{BD}, F, d)$ to the class of all basic type Lie bialgebra structures on $\g$. It then follows from the definition that $X,Y\in Z_a(r_{BD}, F, d)$ are cohomologous if and only if $\mathcal F'(X)$ and $\mathcal F'(Y)$ are equivalent over $F$. This proves the following.

\begin{Prp} Assume that $r_{BD}$ is a Belavin--Drinfeld $r$-matrix with associated admissible triple $(\Gamma_1,\Gamma_2,\tau)$ satisfying the conditions of Proposition \ref{pbd}. Then there is a one-to-one correspondence between $H_a^1(r_{BD},F,d)$ and $F$-equivalence classes of basic Lie bialgebra structures $\delta$ on $\g$ such that $\delta_{\fb}$ is $\fb$-equivalent to $r_{BD}$.
\end{Prp}

If $X^*X=Y^*Y$ for two cocycles $X$ and $Y$, then $Q=YX^{-1}$ satisfies $QX=Y$ and $Q^*Q=I$. Thus any given diagonal $D\in \GL(n,K)$ represents at most one cocycle. The problem of classifying such structures is, in its full generality, beyond the scope of this paper. However, when the admissible triple $(\Gamma_1,\Gamma_2,\tau)$ associated to $r_{BD}$ is trivial (i.e.\ satisfies $\Gamma_1=\Gamma_2=\emptyset$), as is the case when $r_{BD}=r_{DJ}$, then the corresponding cohomology is small, as made precise by the following.

\begin{Prp} Assume that $r_{BD}$ is a Belavin--Drinfeld $r$-matrix associated to the trivial admissible triple, and such that $H_a^1(r_{BD},F,d)\neq\emptyset$. Then $H_a^1(r_{BD},F,d)$ consists of precisely one element if $n$ is even, and is in bijection to $F^*/N(K^*)$ if $n$ is odd.
\end{Prp}

\begin{proof} Assume that $X^*X=SD_X$ with $D_X=\diag(d_1,\ldots d_n)$. Applying $^*$ to both sides leaves the left hand side invariant, while the right hand side becomes $\overline{D_X}S$. Thus $\overline{D_X}=SD_XS$, which implies that $\overline{d_i}=d_{n+1-i}$ for all $i$. Let now $D$ be the diagonal matrix with $D_{ii}=d_i^{-1}$ for all $i\leq n/2$ and $D_{ii}=1$ otherwise. Then $D\in C(r_{BD})$ since the admissible triple associated to $r_{BD}$ is trivial, whence $X$ is cohomologous to $Y=XD$. On the other hand,
\[Y^*Y=\overline{D}X^*XD=\overline{D}SD_XD.\]

If $n$ is even, then a computation shows that the right hand side is $S$. Thus the identity represents the class of $X$, and since $X$ was an arbitrary cocycle, this proves the claim in this case.

If $n$ is odd, the right hand side equals $SD(a)$ where $D(a)$ is the diagonal matrix $\diag(1,\ldots,1,a,1,\ldots,1)$, where the entry in position $(n+1)/2$ is $a=d_{(n+1)/2}\in F^*$. If $Z=QYD$ satisfies $Z^*Z=SD(b)$ for some $b\in F^*$, then in particular $b=N(c)a$ where $c$ is the middle element on the diagonal of $D$. Thus $b\in N(K^*)a$. Conversely if $b=N(c)a$ for some $c\in K$, then $Z=YD(c)$ satisfies $Z^*Z=D(b)$. This proves the statement in the case where $n$ is odd.
 \end{proof}

If the admissible triple associated to $r_{BD}$ satisfies the conditions of Proposition \ref{pbd}, then the set $H_a^1(r_{BD},F,d)$ may indeed be non-empty, even when the admissible triple is non-trivial. This follows from the following.

\begin{Prp} If $-1$ and 2 are squares in $F$, then there exists $X\in\GL(n,K)$ with $X^*X=S$. 
\end{Prp}

\begin{proof} Let  $Y$ be the matrix with entries $Y_{i,i}=Y_{i,n+1-i}=1$ for $i\leq\frac{n+1}{2}$, and $-Y_{i,i}=Y_{i,n+1-i}=\sqrt{-1}$ for $i>\frac{n+1}{2}$, and with all other entries equal to zero. If $n$ is even, then $X=\frac{1}{\sqrt{2}}Y$ satisfies $X^*X=S$. If $n$ is odd, then $X^*X=S$ is satisfied by $X=\frac{1}{\sqrt{2}}D(\sqrt{2})Y$, where $D(\sqrt{2})$ is the diagonal matrix $\diag(1,\ldots,1,\sqrt{2},1,\ldots,1)$, where $\sqrt{2}$ is in position $(n+1)/2$.
\end{proof}

\section{The Twisted Case}
By Proposition \ref{Plambda}, what remains is the case where $r=\sqrt{d'}(\Ad_X\otimes\Ad_X)(r_{BD})$, where $d'\in F^*$ satisfies $\sqrt{d'}\notin K$, $X\in \GL(n,\fb)$ and $r_{BD}$ is a Belavin--Drinfeld $r$-matrix over $\fb$. 

\begin{Prp}\label{nogamma} If $\partial r$ defines a Lie bialgebra structure on $\g$, then the admissible triple $(\Gamma_1,\Gamma_2,\tau)$ associated to $r_{BD}$ is trivial.
\end{Prp}

\begin{proof} From \cite{PS} we know that $(\Gamma_1,\Gamma_2,\tau)$ must satisfy $s(\Gamma_1)=\Gamma_2$ for $\partial r$ to define a Lie bialgebra structure on $\g_K$. From the proof of Proposition \ref{Plambda} we moreover know that $r=r^*$ must hold in order for $\partial r$ to descend to $\g$. Thus
\[\sqrt{d'}(\Ad_{X^*X}\otimes\Ad_{X^*X})(r_{BD})=(\sqrt{d'}(r_{BD}))^*,\]
and using the fact that, by the same proposition, $\sqrt{d'}$ is invariant under $^*$, we get
\[(\Ad_{X^*X}\otimes\Ad_{X^*X})(r_{BD})=r_{BD}^*.\]
Applying $\Ad_S\otimes\Ad_S$ to both sides and arguing as in Proposition \ref{pbd}, the gauge equivalence of $r_{BD}$ and $(\Ad_S\otimes\Ad_S)(r_{BD}^*)$ implies that $s(\Gamma_1)=\Gamma_1$. Altogether we get $\Gamma_1=\Gamma_2$, whence both sets are empty by construction of an admissible triple.
\end{proof}

A classification of those $X$ for which $\partial r$ defines a Lie bialgebra structure on $\g_K$ were given in \cite{PS} and implies the following characterization.

\begin{Prp} The $r$-matrix $r=\sqrt{d'}(\Ad_X\otimes\Ad_X)(r_{DJ})$ defines a Lie bialgebra structure on $\g_K$ if and only if $X=QJD$ for some $Q\in\GL(n,K)$ and $D\in C(r_{DJ})$.
\end{Prp}

The matrix $J$ is given by $J_{i,i}=J_{i,n+1-i}=1$ for $i\leq\frac{n+1}{2}$, $-J_{i,i}=J_{i,n+1-i}=\sqrt{d'}$ for $i>\frac{n+1}{2}$, and with all other entries equal to zero. Note that the triviality of the admissible triple implies that $C(r_{BD})$ consists of all diagonal matrices in $\GL(n,\overline{F})$.

\begin{Rk} The arguments used and quoted in \cite{PS} to derive this result work equally well if $r_{DJ}$ is replaced by $r_s+r_{DJ}$ with $r_s\in\h_{\fb}\wedge\h_{\fb}$ satisfying $\overline{r_s}=-r_s$. The extension is effected by applying Lemma \ref{lcartan}. Thus the proposition applies to any Belavin--Drinfeld $r$-matrix associated to the trivial admissible triple. We omit the proof.
\end{Rk}

Applying the condition that $r$ defines a Lie bialgebra structure on $\g$ gives the following.

\begin{Prp} Let $r_{BD}=r_0+r_1$ be a Belavin--Drinfeld $r$-matrix associated to the trivial admissible triple. The $r$-matrix $r=\sqrt{d'}(\Ad_X\otimes\Ad_X)(r_{BD})$ with $X=QJD$ defines a Lie bialgebra structure on $\g$ if and only if $(\Ad_S\otimes\Ad_S)(r_0)=\overline{r_0}$ and  $J^TQ^*QJ=SD$ for some diagonal $D\in\GL(n,K(\sqrt{d}))$.
\end{Prp}

Here the map $x\mapsto\overline{x}$ is the lift to $\Gal(\fb/F)$ of the involution with respect to $\sqrt{d}$. By Proposition \ref{Plambda}, it leaves $\sqrt{d'}$ invariant. The same holds for the map $x\mapsto x^*$.

\begin{proof} From the proof of Proposition \ref{Plambda} we get $r^*=r$, whence
\[(\Ad_{X^*X}\otimes \Ad_{X^*X})(r_{BD})=r_{BD}^*.\]
The methods used in the proof of Proposition \ref{pbd} are thus applicable. Since the associated triple is trivial, they imply that $\partial r$ defines a Lie bialgebra structure on $\g$ if and only if $(\Ad_S\otimes\Ad_S)(r_0)=\overline{r_0}$ and $X^*X=SD$ for some $D\in C(r_{BD})$. The statement follows by construction of $X$, upon observing that the entries of $D$ belong to $K(\sqrt{d'})$ as this holds for those of the left hand side.
\end{proof}

This allows us to define the following cohomology.

\begin{Def} Let $r_{BD}$ be a Belavin--Drinfeld $r$-matrix with trivial associated admissible triple. We say that $Q\in \GL(n,K)$ is a \emph{compact twisted Belavin--Drinfeld cocycle} if $J^TQ^*QJ=SD_Q$ for some diagonal $D_Q\in \GL(n,K(\sqrt{d'}))$. The set of cocycles is denoted by $\overline{Z}_c(r_{BD}, F, d)$. Two cocycles $Q$ and $R$ are \emph{cohomologous} if $RJ=TQJD$ for some  $T \in\U(n,d)(F)$ and diagonal $D\in \GL(n,K(\sqrt{d'}))$. The set of cohomology classes is denoted by $\overline{H}_c^1(r_{BD},F,d)$.
\end{Def}

As in the non-twisted cases, the following is easy to prove.

\begin{Prp} There is a one-to-one correspondence between $\overline{H}_c^1(r_{BD},F,d)$ and $F$-equivalence classes of basic Lie bialgebra structures $\delta$ on $\g$ such that $\delta_{\fb}$ is $\fb$-equivalent to $r_{BD}$.
\end{Prp}

The computation of the cohomology is however beyond the scope of this paper. 

\bibliographystyle{amsplain}

\begin{thebibliography}{99}
\bibitem{BD} A.\ Belavin and A., V.\ Drinfeld, Triangle equations and simple Lie algebras. Soviet Sci.\ Rev.\ Sect.\ C: Math.\ Phys.\ Rev.\ 4, 93--165 (1984).
\bibitem{Dr} V.\ Drinfeld, Hopf algebras and the quantum Yang--Baxter equation. Dokl.\ Akad.\ Nauk SSSR, 283, 1060--1064 (1985).
\bibitem{EK1} P.\ Etingof and D.\ Kazhdan, Quantization of Lie bialgebras I. Sel.\ Math.\ (NS) 2, 1--41 (1996).
\bibitem{EK2} P.\ Etingof and D.\ Kazhdan, Quantization of Lie bialgebras II. Sel.\ Math.\ (NS) 4, 213--232 (1998).
\bibitem{ES} P.\ Etingof and O.\ Schiffmann, \emph{Lectures on Quantum Groups}. International Press, Somerville, MA (2002).
\bibitem{G} P.\ Gille, Serre's Conjecture II: a survey. In \emph{Quadratic forms, linear algebraic groups and cohomology}, Dev.\ Math.\ 18, 41--56 (2010).
\bibitem{GSz} P.\ Gille and T.\ Szamuely, \emph{Central simple algebras and Galois cohomology}. Cambridge Studies in Advanced Mathematics 101, Cambridge University Press, Cambridge (2006).
\bibitem{Ji} M.\ Jimbo, A $q$-difference analogue of $U\g$ and the Yang--Baxter equation. Lett.\ Math.\ Phys.\ 10, 63--69 (1985).
\bibitem{4046} B.\ Kadets, E.\ Karolinsky, A.\ Stolin and I.\ Pop, Classification of quantum groups and Belavin-Drinfeld cohomologies. Preprint, arXiv:1303.4046 (2013). 
\bibitem{KKPS} B.\ Kadets, E.\ Karolinsky, A.\ Stolin and I.\ Pop, Quantum groups: from Kulish--Reshetikhin discovery to classification. Zapiski Nauchnyh Seminarov POMI, 433, 186--195 (2015).
\bibitem{KR} P.\ Kulish and N.\ Reshetikhin, Quantum Linear Problem for the Sine-Gordon Equation and Higher Representations. Zapiski Nauchnyh Seminarov LOMI 101, 101--110 (1981). Translated to English in J.\ Sov.\ Math.\ 23, 2435--2441 (1983).
\bibitem{Me} A.\ S.\ Merkurjev, On the norm residue symbol of degree 2. Dokl.\ Akad.\ Nauk SSSR 261, 542--547 (1981).
\bibitem{PfD} A.\ Pfister, Zur Darstellung definieter Funktionen als Summe von Quadraten. Invent.\ Math.\ 4, 229--237 (1967).
\bibitem{PfE} A.\ Pfister, \emph{Quadratic Forms with Applications to Algebraic Geometry and Topology}. LMS Lecture Note Series 217, Cambridge University Press, Cambridge (1995).
\bibitem{Se} J.\ -P.\ Serre, \emph{Cohomologie Galoisienne}. Lecture Notes in Mathematics 5, Springer-Verlag, Berlin (1994).
\bibitem{PS} A.\ Stolin, I.\ Pop, Classification of quantum groups and Lie bialgebra structures on $sl(n,F)$. Relations with Brauer group. Preprint, arXiv:1402.3083, (2014).
\end{thebibliography}

\end{document}